\documentclass[a4paper,leqno, oneside, 11pt]{amsart}
\usepackage[top=2cm,bottom=2cm,left=3cm,right=3cm]{geometry}
\usepackage[OT1,T2A]{fontenc}

\usepackage{comment}
\usepackage[english]{babel}
\usepackage{amsmath,amssymb,amsthm,amscd}
\usepackage{latexsym}
\usepackage{mathrsfs}
\usepackage{array}
\usepackage[dvipdfmx]{graphicx}
\usepackage{color}
\usepackage{mathtools}
\usepackage{tikz-cd} 
\usepackage{url}
\usepackage{todonotes}
\usepackage{booktabs} 

\numberwithin{equation}{section}

\DeclareMathOperator{\Spec}{Spec}

\DeclareMathOperator{\End}{End}

\DeclareMathOperator{\GL}{GL}
\DeclareMathOperator{\Res}{Res} 
\DeclareMathOperator{\Gr}{Gr} 

\DeclareMathOperator{\Ker}{Ker}
\DeclareMathOperator{\Img}{Im}
\DeclareMathOperator{\Aut}{Aut}
\DeclareMathOperator{\Inn}{Inn}
\DeclareMathOperator{\Out}{Out}
\DeclareMathOperator{\Lie}{Lie}
\DeclareMathOperator{\Der}{Der} 
\DeclareMathOperator{\ODer}{OutDer} 

\DeclareMathOperator{\Gal}{Gal}

\newcommand{\mbC}{\mathbb{C}}
\newcommand{\mbF}{\mathbb{F}}
\newcommand{\mcF}{\mathcal{F}}
\newcommand{\mbG}{\mathbb{G}}
\newcommand{\mcG}{\mathcal{G}}
\newcommand{\mcU}{\mathcal{U}}
\newcommand{\mcO}{\mathcal{O}}
\newcommand{\mbZ}{\mathbb{Z}}
\newcommand{\mcP}{\mathcal{P}}
\newcommand{\mbP}{\mathbb{P}}
\newcommand{\mbQ}{\mathbb{Q}}

\newcommand{\mfg}{\mathfrak{g}}
\newcommand{\mfu}{\mathfrak{u}}
\newcommand{\p}{\mathfrak{p}}
\newcommand{\mfr}{\mathfrak{r}}

\newcommand{\et}{\mathrm{\acute{e}t}}  
\newcommand{\un}{\mathrm{un}} 
\newcommand{\ab}{\mathrm{ab}} 
\newcommand{\cts}{\mathrm{cts}} 

\definecolor{e-mail}{rgb}{0,.40,.80}
\definecolor{reference}{rgb}{.20,.60,.22}
\definecolor{mrnumber}{rgb}{.80,.40,0}
\definecolor{citation}{rgb}{0,.40,.80}
\usepackage[breaklinks=true,colorlinks=true,linkcolor=reference, citecolor=citation, urlcolor=e-mail]{hyperref}

\theoremstyle{plain}

\newtheorem{theorem}{Theorem}[section]
\newtheorem{lemma}[theorem]{Lemma}
\newtheorem{proposition}[theorem]{Proposition}
\newtheorem{corollary}[theorem]{Corollary}
\newtheorem{conjecture}[theorem]{Conjecture}
\newtheorem{formerconjecture}[theorem]{(Former) Conjecture}
\theoremstyle{definition} 
\newtheorem{definition}[theorem]{Definition}
\newtheorem{example}[theorem]{Example}
\newtheorem{example-definition}[theorem]{Example--Definition}
\newtheorem{remark}[theorem]{Remark}

\title{On graded Lie algebras associated to once-punctured elliptic curves with complex multiplication}
\author[S.~Ishii]{Shun Ishii}
\address{Research Institute for Mathematical Sciences, Kyoto University, Kyoto 606-8502, Japan}
\email{ishii@kurims.kyoto-u.ac.jp}
\date{}

\begin{document}
\begin{abstract}
    We study a graded Lie algebra arising from the Galois action on the pro-$p$ fundamental group of a once-punctured elliptic curve with complex multiplication. Among other things, we provide a minimal generating set of the rationalized Lie algebra under suitable assumptions. The proof is based on a slight variant of the theory of weighted completion of profinite groups developed by Hain and Matsumoto. 
\end{abstract}

\maketitle

\setcounter{tocdepth}{1} 
\tableofcontents

\section{Introduction and main result}\label{sec:1}

The aim of the present paper is to study the structure of a certain Lie algebra arising from the Galois action on the pro-$p$ fundamental group of a once-punctured elliptic curve with complex multiplication. Among other things, we provide a minimal generating set of this Lie algebra under suitable assumptions, using (a slight variant of) the theory of weighted completion developed by Hain and Matsumoto \cite{HM03}. 

\subsection{Lie algebra arising from Galois action on fundamental group}\label{sec:1.1}

Let $p>3$ be a prime, $K$ an imaginary quadratic field of class number one, and fix an algebraic closure $\bar{K}$ of $K$. We write $G_{K}$ for the absolute Galois group $\Gal(\bar{K}/K)$ of $K$. Throughout this section, we assume that $p$ \emph{splits} into two primes in $K$. Let $\p$ denote one of those primes and write $\bar{\p}$ for its conjugate.

Let $E$ be an elliptic curve over $K$ having complex multiplication by the ring of integers $\mcO_{K}$ of $K$. The associated \emph{once-punctured elliptic curve} $X$ is simply defined to be the complement of the origin of $E$. 

We fix a Weierstrass form of $E$ over $K$. It determines the Weierstrass tangential basepoint at the origin \cite[page 209, Case 2]{Na99}. We write $\pi^{\et}_{1}(X)$ for the \'etale fundamental group of $X$ with respect to the Weierstrass tangential basepoint. We have the \'etale homotopy exact sequence 
\[
1 \to \pi_{1}^{\et}(X_{\bar{K}}) \to \pi_{1}^{\et}(X) \to G_{K} \to 1, 
\]
where $X_{\bar{K}} \coloneqq X \times_{\Spec(K)} \Spec(\bar{K})$. The sequence determines the following continuous homomorphism, which we call the \emph{outer Galois representation} (associated to $X$):
\[
\rho_{X} \colon  G_{K} \to \Out(\pi_{1}^{\et}(X_{\bar{K}})) \coloneqq \Aut(\pi_{1}^{\et}(X_{\bar{K}}))/\Inn(\pi_{1}^{\et}(X_{\bar{K}})).   
\]
Here, we note that, since $\pi_{1}^{\et}(X_{\bar{K}})$ is topologically finitely generated, its outer automorphism group has a structure of a profinite group given by congruence topology. The outer Galois representation does not depend on the choice of basepoint. In the present paper, we consider the outer action on the maximal pro-$p$ quotient $\pi_{1}^{(p)}(X_{\bar{K}})$ of $\pi_{1}^{\et}(X_{\bar{K}})$,  which we call the \emph{pro-$p$ outer Galois representation}: 
\begin{align*}
    \rho_{X,p} \colon  G_{K} \to \Out(\pi_{1}^{(p)}(X_{\bar{K}})).
\end{align*} 
We study the outer pro-$p$ Galois representation $\rho_{X,p}$ by studying the associated graded Lie algebra defined by the lower central series on $\pi_{1}^{(p)}(X_{\bar{K}})$. In the rest of this section, we abbreviate $\pi_{1}^{(p)}(X_{\bar{K}})$ as $\pi_{1}$ for simplicity. We write
\[
L^{1}\pi_{1}=\pi_{1} \supset L^{2}\pi_{1}=[\pi_{1}, \pi_{1}] \supset \cdots 
\supset L^{m+1}\pi_{1} = [\pi_{1}, L^{m}\pi_{1}] \supset \cdots
\] for the lower central series of $\pi_{1}$. It induces a filtration on $G_{K}$ by setting
\begin{align*}
F^{m}G_{K} \coloneqq \Ker(G_{K} \xrightarrow{\rho_{X,p}} \Out(\pi_{1}) \to \Out(\pi_{1}/L^{m+1}\pi_{1})) 
\end{align*} for each $m>0$, where the map $\Out(\pi_{1}) \to \Out(\pi_{1}/L^{m+1}\pi_{1})$ is the natural one. This filtration is called the \emph{weight filtration} on $G_{K}$ (associated to $X$ and $p$), as the conjugation action of $G_{K}$ on the $m$-th graded quotient
\[
\mfg_{X,m} \coloneqq \Gr_{F}^{m}G_{K} \coloneqq F^{m}G_{K}/F^{m+1}G_{K}
\] has weight $-m$. Moreover, by \cite[Theorem 2]{As95}, the intersection of the filtration $\{ F^{m}G_{K} \}_{m>0}$ coincides with $\Ker(\rho_{X,p})$.

Since the weight filtration is descending and central, the direct sum
\begin{align*}
\mfg_{X} \coloneqq  \bigoplus_{m>0} \mfg_{X, m}
\end{align*} has a Lie bracket induced by commutators, making it a graded Lie algebra over $\mbZ_{p}$. Moreover, it can be regarded as a Lie subalgebra of the derivation algebra of the graded Lie algebra associated to lower central series of $\pi_{1}$ in a natural way.

Here we remark that $\mfg_{X} \otimes \mbQ_{p}$ does not depend on the choice of $X$:

\begin{lemma}\label{lmm:independence}
    Let $E'$ be another elliptic curve over a finite extension of $K$ having complex multiplication by $\mcO_{K}$, and $\mfg_{X'}$ the associated graded Lie algebra. Then $\mfg_{X'} \otimes \mbQ_{p}$ is isomorphic to $\mfg_{X} \otimes \mbQ_{p}$.
\end{lemma}
\begin{proof}
   Observe that, although $\mfg_{X}$ may change if we replace the base field $K$ with an arbitrary finite extension of $K$, $\mfg_{X} \otimes \mbQ_{p}$ does not. Now the claim follows since $E$ and $E'$ become isomorphic over a finite extension of $K$ by \cite[Theorem 4.3]{Si94}.
\end{proof}

In the present paper, we provide a minimal generating set of $\mfg_{X} \otimes \mbQ_{p}$ under suitable assumptions. We remark that if, moreover, $\mfg_{X} \otimes \mbQ_{p}$ is \emph{freely generated} by such elements and $p$ satisfies certain assumptions, we can determine the kernel of $\rho_{X,p}$ explicitly, cf. Theorem \ref{thm:Ker} (see also Remark \ref{rmk:nonfree}). This result illustrates the importance of understanding the structure of this Lie algebra. We note that these results are analogues of the corresponding results for the projective line minus three points, which we will recall in \textsection \ref{sec:1.3}. 

\subsection{Structure of Lie algebra and its application}\label{sec:1.2}
 
Define a set of pairs of integers by
\begin{align*}
I_{K} \coloneqq \{ \boldsymbol{m}=(m_{1}, m_{2}) \in \mbZ_{>0}^{2} \mid m_{1} \equiv m_{2} \bmod w_{K} \} \setminus \{ (1,1) \},
\end{align*} where $w_{K}$ is the number of roots of unity in $K$ (which is either $2$, $4$ or $6$). It is known (see \cite[(4.2) Proposition]{Na95} and \cite[Proposition 2.2]{Is23+}) that $\mfg_{X,m}$ is trivial whenever $m$ is odd, and is a free $\mbZ_{p}$-module of finite rank on which $G_{K}/F^{1}G_{K}=\Gal(K(E[p^{\infty}])/K)$ acts by conjugation. Moreover, $\mfg_{X,m} \otimes \mbQ_{p}$ decomposes as
\[
\mfg_{X,m} \otimes \mbQ_{p} \cong \bigoplus_{\substack{\boldsymbol{m} \in I_{K}, \\ |\boldsymbol{m}|=m}} \mbQ_{p}(\boldsymbol{m})^{r_{\boldsymbol{m}}} 
\quad \text{for some $r_{\boldsymbol{m}} \in \mbZ_{\geq 0}$}.
\] Here, we write $|(m_{1},m_{2})| \coloneqq m_{1}+m_{2}$ for $(m_{1},m_{2}) \in \mbZ^{2}$. We refer the reader to \textsection \ref{sec:1.5} for the definition of a one-dimensional representation  $\mbQ_{p}(\boldsymbol{m})$, which appears in the $|\boldsymbol{m}|$-th symmetric product of the rational $p$-adic Tate module $V_{p}(E)$. As a consequence, $\mfg_{X} \otimes \mbQ_{p}$ is a bigraded Lie algebra whose bigrading is given by $I_{K}$. The main result of the present paper is stated as follows:

\begin{theorem}\label{thm:main1}
    Suppose the second \'etale cohomology group
    \[
    H^{2}_{\et}(\mcO_{K}[1/p], \mbQ_{p}(\boldsymbol{m}))=\left(\varprojlim_{n}  H^{2}_{\et}(\mcO_{K}[1/p], \mbZ/p^{n}\mbZ(\boldsymbol{m})) \right) \otimes \mbQ_{p}
    \] of the spectrum of the $p$-integer ring of $K$ vanishes for every $\boldsymbol{m} \in I_{K}$. Then the bigraded Lie algebra $\mfg_{X} \otimes \mbQ_{p}$ is generated by $\{ \sigma_{\boldsymbol{m}} \}_{\boldsymbol{m} \in I_{K}}$, where $\sigma_{\boldsymbol{m}}$ is a nonzero element of the $\mbQ_{p}(\boldsymbol{m})$-isotypic component of $\mfg_{X,|\boldsymbol{m}|} \otimes \mbQ_{p}$.
\end{theorem}

An analogous statement holds for inert primes;  see \textsection \ref{sec:3.4} for its statement (Theorem \ref{thm:main2}). The \'etale cohomology group $H^{i}_{\et}(\mcO_{K}[1/p], \mbQ_{p}(\boldsymbol{m}))$ is isomorphic to the $i$-th cohomology group of the Galois group of the maximal $p$-ramified extension of $K$ by \cite[Chapter I\hspace{-.01em}I, Proposition 2.9]{Mi06}. 

The proof of Theorem \ref{thm:main1}, which will be presented in \textsection \ref{sec:3}, is based on (a slight variant of) Hain--Matsumoto's theory of weighted completion of profinite groups.  Roughly speaking, the weighted completion gives a proalgebraic group that approximates the image of $\rho_{X,p}$. Moreover, the cohomology groups of the Lie algebra of its prounipotent radical can be calculated using \'etale cohomology groups of $\Spec(\mcO_{K}[1/p])$. 

\medskip

We have two remarks concerning the assumption on the vanishing of $H^{2}$.

\begin{remark}\label{rmk:EPchar}
    We have the following equivalence for every $(m_{1},m_{2}) \in I_{K}$:
   \[
    H^{2}_{\et}(\mcO_{K}[1/p], \mbQ_{p}(\boldsymbol{m}))=0
    \quad \Leftrightarrow \quad
    \dim_{\mbQ_{p}}  H^{1}_{\et}(\mcO_{K}[1/p], \mbQ_{p}(\boldsymbol{m}))=1.
   \] To obtain this equivalence, it suffices to show that the Euler--Poincar\'e characteristic for $\mbQ_{p}(\boldsymbol{m})$ is equal to $-1$. This follows since the Euler--Poincar\'e characteristic for $\mbQ_{p}(\boldsymbol{m})$ coincides with that of $\mbF_{p}(\boldsymbol{m})$ \cite[A.1.8. Lemme]{PR95}, and the latter is equal to $-1$ by Tate's global Euler--Poincar\'e characteristic formula \cite[(8.7.6)]{NSW08}. Moreover, by Kummer theory, we have  
   \[
   \dim_{\mbQ_{p}} H^{1}_{\et}(\mcO_{K}[1/p], \mbQ_{p}(1))=2.
   \] 
\end{remark}

\begin{remark}\label{rmk:Jannsen}
    Jannsen's conjecture \cite[Conjecture 1]{Ja89} implies that the second cohomology groups in question are trivial.  The group $H^{2}_{\et}(\mcO_{K}[1/p], \mbQ_{p}(\boldsymbol{m}))$ is known to be trivial for every $\boldsymbol{m}=(m,m) \in I_{K}$ by a result of Soul\'e \cite[page 376]{So81},  and for every $\boldsymbol{m}=(m_{1},m_{2}) \in I_{K}$ such that both $m_{1}$ and $m_{2}$ are divisible by $p-1$ \cite[Lemma 4.7]{Is25}. 
\end{remark}

The following lemma gives a sufficient condition for the desired vanishing. Its hypotheses hold, for example, for $(K,p)=(\mbQ(\sqrt{-1}),5), (\mbQ(\sqrt{-3}), 7)$ and $(\mbQ(\sqrt{-19}),7)$.

\begin{lemma}\label{lmm:H2}
    Suppose the class number of $K(p)$ is prime to $p$ and there are only two primes of $K(p^{2})$ lying above $p$. Then $H^{2}_{\et}(\mcO_{K}[1/p], \mbQ_{p}(\boldsymbol{m}))$ is trivial for every $\boldsymbol{m}=(m_{1}, m_{2}) \in I_{K}$.
\end{lemma}

Here, for an ideal $\mathfrak{m}$ of $\mcO_{K}$, we write $K(\mathfrak{m})$ for the ray class field of $K$ of conductor $\mathfrak{m}$.

\begin{proof}
    If $m_{1}$ is congruent to $m_{2}$ modulo $p-1$, the assertion follows from \cite[Corollary A.2]{Is23+}. Hence we may assume $m_{1}$ is not. Then the proof of \cite[Lemma 4.3]{Is23+} shows that $\dim H^{1}_{\et}(\mcO_{K}[1/p], \mbF_{p}(\boldsymbol{m}))=1$. It follows that $H^{2}_{\et}(\mcO_{K}[1/p], \mbF_{p}(\boldsymbol{m}))$ is trivial since the Euler--Poincar\'e characteristic of $\mbF_{p}(\boldsymbol{m})$ is equal to $-1$. Hence $H^{2}_{\et}(\mcO_{K}[1/p], \mbZ_{p}(\boldsymbol{m}))$ vanishes by considering the long exact sequence associated to 
    \[
    0 \to \mbF_{p}(\boldsymbol{m}) \to \mbZ/p^{n+1}\mbZ(\boldsymbol{m}) \to \mbZ/p^{n}\mbZ(\boldsymbol{m}) \to 0
    \] for every $n>0$ and $H^{3}_{\et}(\mcO_{K}[1/p], \mbF_{p}(\boldsymbol{m}))=(0)$ \cite[Chapter V\hspace{-.01em}I\hspace{-.15em}I\hspace{-.15em}I, (8.3.18) Proposition]{NSW08}.    
\end{proof}

The statement of Theorem \ref{thm:main1} can be rephrased in a slightly different way, as follows: For each $\boldsymbol{m} \in I_{K}$, the $G_{K}$-action on the maximal metabelian quotient of $\pi_{1}$ gives rise to a homomorphism \cite[(3.11.4)]{Na95} 
\begin{align*}
    \kappa_{\boldsymbol{m}} \colon \mfg_{X,|\boldsymbol{m}|} \to \mbZ_{p}(\boldsymbol{m}).
\end{align*} This homomorphism was originally introduced by Nakamura \cite{Na95} for general once-punctured elliptic curves over number fields, where these characters were used to study anabelian geometry of once-punctured elliptic curves.  Later, we studied these characters in the case of once-punctured elliptic curves with complex multiplication \cite{Is25}. Among other things, we proved that, under the hypothesis of Theorem \ref{thm:main1}, the character $\kappa_{\boldsymbol{m}}$ is nontrivial for every $\boldsymbol{m} \in I_{K}$ (\cite[Theorem 1.5 (1)]{Is25}, together with \cite[Lemma 2.10]{Is23+}). Since $\kappa_{\boldsymbol{m}}$ vanishes on the commutator Lie ideal, we obtain a homomorphism
\[
(\mfg_{X} \otimes \mbQ_{p})^{\ab} 
\xrightarrow{\oplus \kappa_{\boldsymbol{m}}}
\bigoplus_{\boldsymbol{m} \in I_{K}} \mbQ_{p}(\boldsymbol{m}).
\] Theorem \ref{thm:main1} states that this is an isomorphism, and hence $\{ \sigma_{\boldsymbol{m}} \}_{ \boldsymbol{m} \in I_{K} } $ gives a \emph{minimal} generating set  of $\mfg_{X} \otimes \mbQ_{p}$.

\medskip

We make two more remarks concerning the structure of $\mfg_{X} \otimes \mbQ_{p}$.

\begin{remark}\label{rmk:upperbound}
    Suppose that $K$ is neither $\mbQ(\sqrt{-1})$ nor $\mbQ(\sqrt{-3})$. Theorem \ref{thm:main1} gives a conditional upper bound on the dimension of $\mfg_{X,m} \otimes \mbQ_{p}$ for every $m$. For example, we have the following table:

    \begin{table}[htbp]
        \begin{tabular}{lcccccccc}
            \toprule
            $m$   & 2 & 4 & 6 & 8 & 10 & 12 & 14 & 16 \\
            \midrule
            $\dim (\mfg_{X,m} \otimes \mbQ_{p})$ & 0 & $\leq 3$ & $\leq 5$ & $\leq 10$ & $\leq 24$ & $\leq 50$ & $\leq 120$ & $\leq 270$ \\
            \bottomrule
        \end{tabular}
    \end{table}

   For the two exceptional fields, the conditional upper bounds are given as follows:

   \begin{table}[htbp]
  \centering
  \begin{tabular}{ccccccccc}
    \toprule
    $m$
      & $2$ & $4$ & $6$ & $8$ & $10$ & $12$ & $14$ & $16$ \\
    \midrule
    $\mathbb{Q}(\sqrt{-1})$
      & $0$ & $\leq 1$ & $\leq 3$ & $\leq 3$
      & $\leq 8$ & $\leq 11$ & $\leq 24$ & $\leq 42$ \\
      \midrule
    $\mathbb{Q}(\sqrt{-3})$
      & $0$ & $\leq 1$ & $\leq 1$ & $\leq 3$
      & $\leq 4$ & $\leq 6$ & $\leq 12$ & $\leq 18$ \\
    \bottomrule
  \end{tabular}
  \label{tab:exceptional-dimension-bounds}
\end{table}

   The sequence $3,5,10,24, 50, \dots$ also appears as the sequence of dimensions of the graded components of the \emph{Yang--Mills Lie algebra} $\mathfrak{ym}(3)$ with three generators introduced by Connes and Dubois--Violette \cite{CDV02}, and is  listed in OEIS \cite[A072337]{OEIS}. This coincidence is explained by the fact that the Lie ideal $\oplus_{m>2} \Gr^{m}\mathfrak{ym}(3)$ is freely generated by one element in each degree indexed by $I_{K}$. 
    
\end{remark}

\begin{remark}\label{rmk:stable}
    Tsunogai \cite[page 223]{Ts03} introduced a certain graded Lie algebra over $\mbQ$, called the \emph{stable derivation algebra of genus one}. The graded Lie algebra associated with the pro-$p$ outer Galois representation of an arbitrary once-punctured elliptic curve over a number field is embedded into this algebra (over $\mbQ_{p}$).  The dimensions of its graded components of degree up to $12$ are given as follows \cite[Table 1 of \textsection 3.11 and Proposition 3.15]{Ts03}:
\medskip

\centerline{%
\begin{tabular}{lcccccc}
\toprule
$m$   & 2 & 4 & 6 & 8 & 10 & 12 \\
\midrule
$\dim$ & 0 & 3 & 6 & 10 & 25 & 50 \\
\bottomrule
\end{tabular}}

\medskip
    For example, the dimension at degree $6$ differs by one from the upper bound in Remark \ref{rmk:upperbound}. Roughly speaking, this discrepancy arises because two independent elements in the dual of the degree $6$ component of the stable derivation algebra become linearly dependent on the dual of $\mfg_{X,6} \otimes \mbQ_{p}$ by an  explicit formula \cite[Theorem 1.7]{IG25} relating the Soul\'e character $\kappa_{m}$ (see \textsection \ref{sec:1.3}) with its elliptic counterpart $\kappa_{(m,m)}$.  On the other hand, for a generic once-punctured elliptic curve over $\mbQ$, Tsunogai has shown that the associated Lie algebra coincides with the stable derivation algebra of genus one, at least up to degree $12$ \cite[Proposition 3.15]{Ts03}. 
\end{remark}

In our previous work, we conjectured that $\mfg_{X} \otimes \mbQ_{p}$ is free, which we have recently found to be false unless $K=\mbQ(\sqrt{-1})$ or $\mathbb{Q}(\sqrt{-3})$; see Remark \ref{rmk:nonfree}.

\begin{formerconjecture}[{\cite[Conjecture 2.3]{Is23+}}]\label{cnj:ellDI}
     The Lie algebra  $\mfg_{X} \otimes \mbQ_{p}$ is a free bigraded Lie algebra with one generator in the $\mbQ_{p}(\boldsymbol{m})$-isotypic component of $\mfg_{X, |\boldsymbol{m}|} \otimes \mbQ_{p}$ for each $\boldsymbol{m} \in I_{K}$.
\end{formerconjecture}

\begin{remark}\label{rmk:nonfree}
    After completing the manuscript, we realized that Conjecture \ref{cnj:ellDI} \emph{does not hold} unless $K=\mbQ(\sqrt{-1})$ or $K=\mathbb{Q}(\sqrt{-3})$. Since details lie beyond the scope of the present paper, we defer them to a separate paper. 

    Let $m>1$ be an odd integer. First, it is easy to show that the $\mbQ_{p}(m,1)$-component of $\mathfrak{g}_{X} \otimes \mbQ_{p}$ is at most one-dimensional, and is zero if and only if $\kappa_{(m,1)}$ is trivial. We fix a generator $\sigma_{(m,1)}$ of this component (which may be zero).

    Now suppose that $K \not \in \{ \mbQ(\sqrt{-1}), \mbQ(\sqrt{-3}) \}$. For example, we have found that $[\sigma_{(3,1)}, \sigma_{(9,1)}]$ and $[\sigma_{(5,1)}, \sigma_{(7,1)}]$ are linearly dependent. This relation coincides with one of Pollack's relations \cite[(4)]{Po09} in the derivation algebra of a free Lie algebra of rank two, up to a scalar. Here, we identify a free Lie algebra of rank two over $\mbQ_{p}$ with the Lie algebra associated to the lower central series of $\pi_{1}$ tensored with $\mbQ_{p}$. Similarly, the other Pollack's relations impose constraints on the structure of $\mfg_{X} \otimes \mbQ_{p}$.

    In the two exceptional cases, however, since the Lie algebras in question do not have a nontrivial $\mbQ_{p}(3,1)$-component, we have $\sigma_{(3,1)}=0$. In general, one can show that Pollack's relations do not give rise to nontrivial relations in the Lie algebra $\mfg_{X} \otimes \mbQ_{p}$  when $K \in \{ \mbQ(\sqrt{-1}), \mbQ(\sqrt{-3}) \}$. We still believe that Conjecture \ref{cnj:ellDI} holds in  these two cases for this reason. 
\end{remark}

One motivation for considering Conjecture \ref{cnj:ellDI} is the following theorem, which determines the kernel of the pro-$p$ outer Galois representation associated to $X$ explicitly:

\begin{theorem}[{\cite[Theorem 2.14]{Is23+}}]\label{thm:Ker}
    Suppose that
    \begin{enumerate}
        \item the class number of $K(p)$ is prime to $p$,
        \item there are only two primes of $K(p^{2})$ lying above $p$, and
        \item Conjecture \ref{cnj:ellDI} holds.
    \end{enumerate}
    Then, the fixed field $\bar{K}^{\Ker(\rho_{X,p})}$ of the kernel of $\rho_{X,p}$ coincides with the composite field of $K(E[p])$ and the maximal pro-$p$ extension of $K(p)$ unramified outside $p$.
\end{theorem}

Theorem \ref{thm:Ker} may be regarded as an analogue of Sharifi's result \cite{Sh02} (Theorem \ref{thm:Sharifi}) in the case of the projective line minus three points. In this case, the fixed field of the kernel coincides with the maximal pro-$p$ extension of $\mbQ(\mu_{p})$ unramified outside $p$ whenever $p>2$ is regular.

\subsection{Case of the projective line minus three points}\label{sec:1.3}

In this subsection, we briefly review several previous results in the case of the thrice-punctured projective line $T \coloneqq \mathbb{P}^{1}_{\mbQ}-\{ 0,1,\infty \}$. Let 
\[
\rho_{T,p} \colon G_{\mbQ}=\Gal(\bar{\mbQ}/\mbQ)
\to \Out(\pi_{1}(T_{\bar{\mbQ}})^{(p)})
\] be the pro-$p$ outer Galois representation. As was explained in \textsection \ref{sec:1.1}, the lower central series on the pro-$p$ fundamental group induces the weight filtration on $G_{\mbQ}$. We write $\mathfrak{t}$ for the associated graded Lie algebra
$
    \mathfrak{t} \coloneqq \bigoplus_{m>0} \mathfrak{t}_{m} 
$ over $\mbZ_{p}$. It is known that $\mathfrak{t}_{m}$ is isomorphic to a finite sum of the $m$-th Tate twist $\mbZ_{p}(m)$.  The following conjecture is nowadays often referred to as the \emph{Deligne--Ihara conjecture}:

\begin{conjecture}[proved by Hain--Matsumoto {\cite{HM03}} and Brown {\cite{Br12}}]\label{cnj:DI}
    The graded Lie algebra $\mathfrak{t} \otimes \mbQ_{p}$ is freely generated by one element in each odd degree $>1$.
\end{conjecture}

Hain and Matsumoto \cite[Conjecture 1]{HM03} proved the generation portion of the conjecture using their theory of weighted completion of profinite groups. Later, Brown \cite{Br12} established a certain proalgebraic analogue of Belyi's theorem, stating that the homomorphism
\[
U_{\mathrm{Betti}} \to \Out(\pi_{1}^{\un}(\mbP^{1}(\mbC)-\{ 0,1,\infty \}, \vec{01}))
\] is a closed immersion, as a corollary of his basis theorem for motivic multiple zeta values \cite[Theorem 1.1]{Br12}. This result implies the Deligne--Ihara conjecture. Here, we write $U_{\mathrm{Betti}}$ for the prounipotent radical of the fundamental group of the tannakian category $\mathrm{MTM}(\mbZ)$ of mixed Tate motives over $\mbZ$ with respect to the Betti realization functor, $\pi_{1}^{\un}(\mbP^{1}(\mbC)-\{ 0,1,\infty \}, \vec{01})$ for the unipotent completion of the topological fundamental group of $\mbP^{1}_{\mbC}-\{ 0,1,\infty \}$, and $\Out$ for the outer automorphism group scheme \cite[Corollary A.9]{HM03}. From this motivic point of view, Hain--Matsumoto's theory of constrained weighted completion \cite[9.4]{HM03} gives the  tannakian fundamental group of $\mathrm{MTM}(\mbZ)$ with respect to the $\ell$-adic realization functor.

For each odd integer $m>1$, we have a nontrivial character, called the $m$-th \emph{Soul\'e character}
\[
    \kappa_{m} \colon \mathfrak{t}_{m} \to \mbZ_{p}(m)
\] which arises from the Galois action on the maximal metabelian quotient of $\pi_{1}(T_{\bar{\mbQ}})^{(p)}$, cf. Ihara--Kaneko--Yukinari \cite{IKY87}. We refer the interested reader to Ihara's beautiful survey \cite{Ih02} for several arithmetic properties of Soul\'e characters.  Hain--Matsumoto's result implies that 
\[
    (\mathfrak{t} \otimes \mbQ_{p})^{\ab} \xrightarrow{\oplus \kappa_{m}} \bigoplus_{m>1, \mathrm{odd}} \mbQ_{p}(m)
\] is an isomorphism. Finally, we mention one significant application of Conjecture \ref{cnj:DI}, which gives an affirmative answer to a question posed by Anderson and Ihara \cite[page 272 (a)]{AI88} for every odd regular prime:

\begin{theorem}[Sharifi]\label{thm:Sharifi}
    Assume $p>2$ is regular and Conjecture \ref{cnj:DI} holds. Then the fixed field of $\Ker(\rho_{T,p})$ coincides with the maximal pro-$p$ extension of $\mbQ(\mu_{p})$ unramified outside $p$.
\end{theorem}

\subsection{Notational conventions}\label{sec:1.5}

In this subsection, $p$ denotes an arbitrary prime. 

\subsubsection*{Profinite Groups}\label{Profinite Groups}

Let $G$ be a profinite group.

\begin{itemize}
    \item We write $G^{\mathrm{ab}}$ for the maximal abelian quotient of $G$.
    \item We write $G^{(p)}$ for the maximal pro-$p$ quotient of $G$.
\end{itemize}

\subsubsection*{Fields}\label{Fields}

We fix an algebraic closure $\bar{\mbQ}$ of $\mbQ$ and every number field is considered to be a subfield of $\bar{\mbQ}$.  Suppose $F$ is a subfield of $\bar{\mbQ}$. 

\begin{itemize}
    \item For a positive integer $n$, we write $\mu_{n}$ for the set of $n$-th roots of unity in $\bar{\mbQ}$.
    \item We write $G_{F}$ for the absolute Galois group $\Gal(\bar{\mbQ}/F)$ of $F$.
    \item We write $S_{p}(F)$ for the set of primes of $F$ lying above $p$. We abbreviate it as $S_{p}$ if $F$ is clear from the context.  For $\lambda \in S_{p}(F)$, we write $F_{\lambda}$ for the $\lambda$-adic completion of $F$. We often fix an element of $S_{p}(\bar{\mbQ})$ above $\lambda$ implicitly, and consider a homomorphism $
    G_{F_{\lambda}} \coloneqq \Gal(\bar{\mbQ}_{p}/F_{\lambda}) \to G_{F}
    $ corresponding to that prime. 
\end{itemize}

\subsubsection*{Elliptic curves with complex multiplication}

Let $K$ be an imaginary quadratic field of class number one and $E$ an elliptic curve over $K$ having complex multiplication by $\mcO_{K}$. 

\begin{itemize}
\item We write $K(\mathfrak{m})$ for the ray class field of $K$ of conductor $\mathfrak{m}$ for an integral ideal $\mathfrak{m}$ of $K$. We also write $K(\mathfrak{m}^{\infty})$ for the union $\cup_{n>0} K(\mathfrak{m}^{n})$. Note that the Galois action on the group of $\mathfrak{m}$-torsion points of $E$ determines an injective homomorphism
\[ 
\Gal(K(E[\mathfrak{m}])/K) \hookrightarrow \Aut(E[\mathfrak{m}]) \cong (\mcO_{K}/\mathfrak{m})^{\times}
\] and an isomorphism
\[
    \Gal(K(\mathfrak{m})/K) \xrightarrow{\sim} (\mcO_{K}/\mathfrak{m})^{\times}/\Img(\mcO_{K}^{\times} \to  (\mcO_{K}/\mathfrak{m})^{\times}).
\] The latter isomorphism does not depend on the choice of $E$.
\item Assume $p$ splits in $K$ as $(p)=\p \bar{\p}$. The $p$-adic Tate module $T_{p}(E)$ splits as
\[
T_{p}(E) \cong
\mcO_{K} \otimes \mathbb{Z}_{p} \xrightarrow{\sim} \mcO_{K_{\p}} \times \mcO_{K_{\bar{\p}}}=\mbZ_{p} \times \mbZ_{p}.
\] The $\p$-adic (resp. $\bar{\p}$-adic) Tate module $T_{\p}(E)$ (resp. $T_{\bar{\p}}(E)$) is identified with the first (resp. second) component. We have two characters
\[
\chi_{1} \colon G_{K} \to \Aut(T_{\p}(E))=\mbZ_{p}^{\times}
\quad \text{and} \quad
\chi_{2} \colon  G_{K} \to \Aut(T_{\bar{\p}}(E))=\mbZ_{p}^{\times}.
\] 
We write 
$V_{p}(E) \coloneqq T_{p}(E) \otimes \mbQ_{p}$, 
$V_{\p}(E) \coloneqq T_{\p}(E) \otimes \mbQ_{p}$ and
$V_{\bar{\p}}(E) \coloneqq T_{\bar{\p}}(E) \otimes \mbQ_{p}$.
\item Suppose $\boldsymbol{m}=(m_{1},m_{2}) \in \mbZ^{2}$. We write 
\[
\chi^{\boldsymbol{m}} \coloneqq \chi_{1}^{m_{1}}\chi_{2}^{m_{2}} \colon G_{K} \to \mbZ_{p}^{\times},
\] which factors through $\Gal(K(p^{\infty})/K)$ whenever $m_{1}$ is congruent to $m_{2}$ modulo $w_{K}$. 

\item Let $M$ be a finitely generated $\mbZ_{p}$-module or a finite-dimensional $\mbQ_{p}$-vector space on which $G_{K}$ acts continuously. We write $M(\boldsymbol{m})$ for its $\boldsymbol{m}$-th twist. 
\end{itemize}

\subsubsection*{Indices}\label{Indexes}

Let $S$ be a nonempty finite set and $\boldsymbol{m}=(m_{w})_{w \in S}  \in \mbZ^{S}$. We write
\begin{itemize}
    \item $\boldsymbol{0} \coloneqq (0, \dots, 0, \dots, 0) \in \mbZ^{S}$ for the zero index,
    \item $|\boldsymbol{m}| \coloneqq \sum_{w \in S}m_{w}$ for the sum of all components.
\end{itemize}

Let $V$ be a finite-dimensional $\mbQ_{p}$-vector space on which $\mbG_{m}^{S}$ acts. We write $V_{\boldsymbol{m}}$ for its subspace on which the $w$-component of $\mbG_{m}^{S}$ acts by the $m_{w}$-th power of the standard character for every $w \in S$.

\section{Weighted completion and structure of its Lie algebra}\label{sec:2}

The aim of this section is to review the theory of weighted completion of profinite groups developed by Hain and Matsumoto \cite{HM03} in a slightly broader setting than the original one. More precisely, in the following subsections, we consider an arbitrary nonempty finite set of central cocharacters, whereas they consider one cocharacter in \cite[3]{HM03} when defining the notion of negatively weighted extension (see Definition \ref{dfn:wtdext} below). We omit proofs for several assertions when this difference does not cause any  problems.

\subsection{Negatively weighted extension}\label{sec:2.1}

Let $R$ be a reductive group over $\mbQ_{p}$ and $S=\{ w \colon \mathbb{G}_{m} \to R \}$ a nonempty finite set of central cocharacters. Here, a homomorphism $\mbG_{m} \to R$ is called a \emph{central cocharacter} if its image is contained in the center of $R$. We write $w_{S}$ for the total cocharacter defined by
\begin{align*}
    w_{S} \colon \mbG_{m} \xrightarrow{\mathrm{diag}} \mbG_{m}^{S} \xrightarrow{\prod w} R
\end{align*}
where the first arrow denotes the diagonal embedding.

\begin{definition}\label{dfn:wtdext}
Suppose an extension 
\[
1 \to U \to G \to R \to 1
\] of $R$ by a unipotent group $U$ is given. The first homology of $U$ (i.e. the abelianization of its Lie algebra) is an $R$-module, and hence is a $\mbG_{m}^{S}$-module via $\prod_{w \in S} w \colon \mbG_{m}^{S} \to R$. Therefore, it can be written as 
\[
H_{1}(U)=\bigoplus_{\boldsymbol{m} \in \mbZ^{S}} H_{1}(U)_{\boldsymbol{m}}.
\] If $H_{1}(U)_{\boldsymbol{m}}$ is trivial for every $\boldsymbol{m}  \in \mbZ^{S} \setminus \mbZ_{<0}^{S}$, we say the extension is \emph{negatively weighted with respect to $S$}.
\end{definition}

Let $G$ be a negatively weighted extension of $R$ by a unipotent group $U$ with respect to $S$, and $V$ a finite-dimensional representation of $G$.

In what follows, a \emph{$G$-module} means a finite-dimensional representation of $G$. Choose a section $s \colon R \to G$, which exists and is unique up to conjugation by an element of $U(\mbQ_{p})$ by \cite[V\hspace{-.01em}I\hspace{-.15em}I\hspace{-.15em}I, Theorem 4.3]{Ho81}, for example. For each $w \in S$, we write $\tilde{w} \colon \mbG_{m}\to G$ for a lift of $w$ induced by $s$. Then $V$ can be regarded as a $\mbG_{m}^{S}$-module via $\prod_{w \in S} \tilde{w} \colon \mbG_{m}^{S} \to G$, and we obtain a decomposition
\[
V=\bigoplus_{\boldsymbol{m} \in \mbZ^{S}} V_{\boldsymbol{m}}.
\] We define the \emph{weight filtration} on $V$ by 
\[
W_{n}V \coloneqq \bigoplus_{|\boldsymbol{m}| \leq n} V_{\boldsymbol{m}}
\] for each $n \in \mbZ$, and its $n$-th graded quotient by
\[
\Gr_{n}^{W}V \coloneqq W_{n}V/W_{n-1}V \cong \bigoplus_{|\boldsymbol{m}|=n}  V_{\boldsymbol{m}}
\] where the second isomorphism is induced by the inclusion 
$
\bigoplus_{|\boldsymbol{m}|=n}  V_{\boldsymbol{m}} \hookrightarrow W_{n}V.
$ 
We consider the adjoint action of $G$ on Lie algebras $\mfg \coloneqq \Lie(G)$, $\mfu \coloneqq \Lie(U)$ and $\mfr \coloneqq \Lie(R)$. We then have the following proposition.

\begin{proposition}[cf. {\cite[Proposition 3.4 and Corollary 3.5]{HM03}}]\label{prp:wtdlie}\,  
    \begin{enumerate}
    \item If $V$ is a $G$-module, $x \in \mfg_{\boldsymbol{n}}$ and $v \in V_{\boldsymbol{m}}$, then $[x,v] \in V_{\boldsymbol{n}+\boldsymbol{m}}$.
    \item $\mfg$ is a $\mbZ^{S}$-graded Lie algebra, and $\mfu$ is its $\mbZ^{S}$-graded Lie subalgebra.
    \item $\mfu_{\boldsymbol{n}}=0$ unless $\boldsymbol{n} \in \mbZ_{<0}^{S}$.
    \item The decomposition  $\mfg=\bigoplus_{\boldsymbol{n} \in \mbZ^{S}} \mfg_{\boldsymbol{n}}$ is given by
    \[
    \mfg_{\boldsymbol{n}}=
    \begin{cases}
        \mathfrak{r} & (\boldsymbol{n}=\boldsymbol{0}) \\
        \mfu_{\boldsymbol{n}} & (\boldsymbol{n} \in \mbZ^{S}_{<0}) \\
        0 & (\mathrm{otherwise}).
    \end{cases}
    \]
    \item The weight filtration on $\mfg$ is given by Lie ideals, and we have 
    \[
    \mfg=W_{0} \mfg, \quad
    \mfu=W_{-1} \mfg \quad \text{and} \quad
    \mfr \cong \Gr_{0}^{W} \mfg. 
    \]
    \end{enumerate}
\end{proposition}
\begin{proof}
    To prove the assertions (1)--(4), we may assume that $S$ consists of a single cocharacter, noting that $S$ consists of central cocharacters. Then the corresponding assertions are proved in \cite[Proposition 3.4]{HM03}. The last assertion (5) is the same as \cite[Corollary 3.5]{HM03}.
\end{proof}

The weight filtration on a $G$-module $V$ does not depend on the choice of the section:

\begin{proposition}[{\cite[Proposition 3.8]{HM03}}]\label{prp:wtdfilt}
    Let $V$ be a $G$-module. 
\begin{enumerate}
    \item The weight filtration on a $G$-module $V$ does not depend on the choice of the section $s \colon R \to G$, and each $W_{m}V$ is a $G$-module. 
    \item For each $m \in \mbZ$, the $G$-action on $\Gr_{m}^{W}V$ factors through $R$, and $\mbG_{m}$ acts by the $m$-th power of the standard character via $w_{S}$.
\end{enumerate}
\end{proposition}

As a corollary, $\Gr^{W}_{\bullet} \mfu =\bigoplus_{m \in \mbZ} \Gr^{W}_{m} \mfu$ is a $\mbZ^{S}$-graded Lie algebra, which does not depend on the choice of the section. 

\begin{corollary}
The Lie algebra $\Gr^{W}_{\bullet} \mfu$ is a $\mbZ^{S}$-graded Lie algebra and is independent of the choice of the section. Moreover, its  $\boldsymbol{m}$-th component is given by $(\Gr^{W}_{|\boldsymbol{m}|} \mfu)_{\boldsymbol{m}}$ for each $\boldsymbol{m} \in \mbZ^{S}$, which is trivial unless $\boldsymbol{m} \in \mbZ^{S}_{<0}$.
\end{corollary}

\begin{remark}\label{rmk:proalgebraic}
One can generalize the notion of negatively weighted extension to the case where $G$ is \emph{proalgebraic}. All results extend verbatim to the proalgebraic case by standard limit arguments. One can also generalize the notion of weight filtration on a $G$-module to the case of a projective limit of $G$-modules, such as a Lie algebra of a proalgebraic group $G$ which is a negatively weighted extension of $R$.  
\end{remark}

\subsection{Weighted completion of a profinite group}\label{sec:2.2}

In this subsection, we introduce the notion of weighted completion of a profinite group with certain additional data and explain Theorem \ref{thm:wtdLie} and its corollaries. Throughout this subsection, suppose the following data $(\Gamma,R,S, \rho)$ are given:

\begin{enumerate}
    \item $\Gamma$: a profinite group,
    \item $R$: a reductive algebraic group over $\mbQ_{p}$,
    \item $S= \{w \colon \mathbb{G}_{m} \to R \}$: a nonempty finite set of central cocharacters, and
    \item $\rho \colon \Gamma \to R(\mbQ_{p})$: a continuous homomorphism whose image is Zariski-dense.
\end{enumerate}  

\begin{example}[cf. {\cite[\textsection 7]{HM03}}]\label{ex:MTM}
    One typical example of such data is given as follows.
    \begin{enumerate}
        \item $\Gamma$ is the Galois group of the maximal $p$-ramified extension of $\mbQ$,
        \item $R \coloneqq \mbG_{m}$,
        \item $S$ only consists of a single cocharacter $w \colon   \mbG_{m} \to \mbG_{m}; x \mapsto x^{-2}$, and  
        \item $\rho \colon \Gamma \to R(\mathbb{Q}_{p})=\mbQ_{p}^{\times}$ is the $p$-adic cyclotomic character.
    \end{enumerate}
    Here, we take the inverse of the square of the standard character to make the representation-theoretic weight compatible with the weight defined via Frobenius. 
\end{example}

\begin{example-definition}\label{ex:MCMM}
    Let $E$ be an elliptic curve over an imaginary quadratic field $K$ having complex multiplication by $\mcO_{K}$, and $p$ a prime that splits into two primes in $K$. We write $S_{E,p}$ for a set of primes of $K$ consisting of those at which $E$ has bad reduction, $S_{p}(K)=\{ \p, \bar{\p} \}$ and the set of archimedean primes. We then obtain the following data: 

    \begin{enumerate}
        \item $\Gamma$ is the Galois group of the maximal extension of $K$ unramified outside $S_{E,p}$,
        \item $R \coloneqq \mbG_{m}^{2}$,
        \item $S=\{ w_{\p}, w_{\bar{\p}} \}$ is the set of two central cocharacters of $\mbG_{m}^{2}$ defined by
        \[
         w_{\p}(x)=(x^{-1},1)
         \quad \text{and} \quad
         w_{\bar{\p}}(x)=(1, x^{-1}),
         \] 
        \item $\rho \colon G_{K} \to R(\mbQ_{p})$ is the Galois representation attached to $V_{p}(E)=V_{\p}(E) \oplus V_{\bar{\p}}(E)$.
    \end{enumerate}

    One can construct another set of data as follows: Let $\mu_{K}$ be the group of roots of unity contained in $K$. We have an injective homomorphism 
\[
 \mu_{K} \subset K^{\times} \subset (K \otimes \mbQ_{p})^{\times} \xrightarrow{\sim} K_{\p}^{\times} \times K_{\bar{\p}}^{\times} =R(\mbQ_{p})
\] If we regard $\mu_{K}$ as a (constant) closed subgroup of $\mbG_{m}^{2}$, the quotient $\mathbb{G}_{m}^{2}/\mu_{K}$ is again isomorphic to $\mbG_{m}^{2}$, because the homomorphism
\[
\mbG_{m}^{2} \to \mbG_{m}^{2}; (x,y) \mapsto (xy, (xy^{-1})^{\frac{w_{K}}{2}}) 
\] is faithfully flat and has $\mu_{K}$ as its kernel. The Galois representation
\[
 \Gamma \xrightarrow{\rho_{E,p}}  \mbG_{m}^{2}(\mbQ_{p}) \to (\mbG_{m}^{2}/\mu_{K})(\mbQ_{p})
\] is unramified outside $p$, since the kernel of the representation corresponds to the maximal abelian extension of $K$ unramified outside $p$ \cite[Chapter I\hspace{-1.2pt}I, Theorem 5.6]{Si94}. In particular, it does not depend on the choice of $E$. We obtain the following data, which we refer to as the \emph{data associated with $K$}:
\begin{enumerate}
    \item $\Gamma$ is the Galois group of the maximal extension of $K$ unramified outside $p$,
    \item $R \coloneqq \mathbb{G}_{m}^{2}/\mu_{K}$,
    \item $S$ is the set of two cocharacters induced by $\{ w_{\p}, w_{\bar{\p}} \}$, and
    \item $\rho \colon \Gamma \to R(\mbQ_{p})$ is the representation obtained above.
\end{enumerate}
\end{example-definition}

The \emph{weighted completion} is a universal proalgebraic group $\mcG$ that fits into the exact sequence
\[
1 \to \mcU \to \mcG \to R \to 1
\] which is a negatively weighted extension of $R$ with respect to $S$, together with a lift of $\rho$ to $\mcG$: 

\begin{definition}\label{dfn:wtdcmpl}
    The weighted completion of $(\Gamma, R, S, \rho)$ is a pair $(\mcG, \tilde{\rho})$ that consists of 
    \begin{itemize}
        \item A proalgebraic group $\mathcal{G}$ over $\mbQ_{p}$, which is a negatively weighted extension of $R$ with respect to $S$ by a prounipotent group $\mathcal{U}$, and 
        \item a continuous homomorphism $\tilde{\rho} \colon \Gamma \to \mathcal{G}(\mbQ_{p})$ that lifts $\rho$.
    \end{itemize}

    Moreover, the pair is assumed to satisfy the following universal property: Let $(G, \phi)$ be an arbitrary pair of a proalgebraic group over $\mbQ_{p}$, which is a negatively weighted extension of $R$ with respect to $S$ by a prounipotent group, and a continuous homomorphism $\phi \colon \Gamma \to G(\mbQ_{p})$ that lifts $\rho$. Then there exists a unique homomorphism $\Phi \colon \mathcal{G} \to G$ over $R$ satisfying $\phi=\Phi \circ \tilde{\rho}$.
\end{definition}

The universal property of the weighted completion characterizes the pair $(\mcG, \tilde{\rho})$ up to unique isomorphism. Moreover, an argument similar to that used in the proof of \cite[Proposition 4.4]{HM03} shows that the weighted completion of $(\Gamma, R, S, \rho)$ always exists.

Hain and Matsumoto relate the continuous cohomology groups of the Lie algebra $\mfu$ of the prounipotent radical of the weighted completion of $(\Gamma,R,S,\rho)$ with the continuous cochain cohomology groups of $\Gamma$ when $S$ consists of a single cocharacter \cite[Theorem 4.8 and Theorem 4.9]{HM03}, and their proof is applicable to the setting of the present paper. In the following theorem, we refer the reader to \cite[4.1]{HM03} for the continuous cochain cohomology of $\mfu$ and to \cite[Chapter I\hspace{-.01em}I, \textsection 7]{NSW08} for details on the  continuous cochain cohomology of profinite groups, though we identify them with the \'etale cohomology groups when considering the data associated with $K$ by \cite[Chapter I\hspace{-.01em}I, Proposition 2.9]{Mi06}.

We introduce some notation before stating the theorem. Let $V$ be a nonzero $R$-module, and decompose it into the direct sum 
$
V=\bigoplus_{\boldsymbol{m} \in \mbZ^{S}} V_{\boldsymbol{m}}.
$ Each $V_{\boldsymbol{m}}$ is an $R$-module since our cocharacters are central. If $V=V_{\boldsymbol{n}}$, then $\boldsymbol{n}$ is called the weight of $V$.

\begin{theorem}[cf. {\cite[Theorem 4.8, Theorem 4.9 and Corollary 4.10]{HM03}}]\label{thm:wtdLie}
    Let $\{ V_{\alpha} \}_{\alpha}$ be a set of representatives of the isomorphism classes of finite-dimensional irreducible representations of $R$ and $n(\alpha) \in \mbZ^{S}$ the weight of $V_{\alpha}$. Write $\mfu$ for the Lie algebra of the prounipotent radical of the weighted completion of $(\Gamma, R, S, \rho)$. 
    \begin{enumerate}
        \item If $H^{1}_{\mathrm{cts}}(\Gamma, V_{\alpha})$ is finite-dimensional for every $\alpha$ with $n(\alpha) \in \mbZ^{S}_{<0}$, we have 
    \begin{align*}
        &H^{1}_{\mathrm{cts}}(\mfu) \cong \bigoplus_{n(\alpha) \in \mbZ^{S}_{>0}} H^{1}_{\mathrm{cts}}(\Gamma, V_{\alpha}^{\ast}) \otimes V_{\alpha} \quad \text{and} \\
        &H_{1}(\mfu) \cong \prod_{n(\alpha) \in \mbZ^{S}_{<0}} H^{1}_{\mathrm{cts}}(\Gamma, V_{\alpha})^{\ast} \otimes V_{\alpha}.
    \end{align*} More precisely, $W_{0}H^{1}_{\mathrm{cts}}(\mfu)$ vanishes and there is an $R$-module isomorphism
    \[
    \Gr^{W}_{m} H^{1}_{\cts}(\mfu) \cong \bigoplus_{\substack{n(\alpha) \in \mbZ^{S}_{>0}, \\ |n(\alpha)|=m}} H^{1}_{\mathrm{cts}}(\Gamma, V_{\alpha}^{\ast}) \otimes V_{\alpha}
    \] for each $m>0$.
    \item There is an injective $R$-module homomorphism
     \[
     \Phi \colon
    H^{2}_{\mathrm{cts}}(\mfu) \hookrightarrow \bigoplus_{\substack{n(\alpha) \in \mbZ^{S}_{>1}}} H^{2}_{\mathrm{cts}}(\Gamma, V_{\alpha}^{\ast}) \otimes V_{\alpha}.
    \]  More precisely, we have an injective $R$-module homomorphism
    \[
    \Gr^{W}_{m} H^{2}_{\cts}(\mfu) \hookrightarrow \bigoplus_{\substack{n(\alpha) \in \mbZ^{S}_{>1}, \\ |n(\alpha)|=m}} H^{2}_{\mathrm{cts}}(\Gamma, V_{\alpha}^{\ast}) \otimes V_{\alpha}
    \] for each $m>1$.
    \item Suppose that $H^{2}_{\mathrm{cts}}(\Gamma, V_{\alpha}^{\ast})$ vanishes for every $n(\alpha)\in \mbZ^{S}_{>1}$. Then $\Gr^{W}_{\bullet} \mfu$ is a free $\mbZ^{S}$-graded Lie algebra, and any lift of a basis of $\Gr^{W}_{\bullet} H_{1}(\mfu)$ gives a free generating set.
    \end{enumerate} 
\end{theorem}

We remark that the weight filtration on $\mfu$ passes to its cohomology by \cite[Proposition 4.5]{HM03}. The last assertion follows from the other two, together with \cite[Lemma 4.11]{HM03}.  The proof of the first two assertions is more or less the same as the original one, noting that $H^{n}_{\cts}(\mfu)$ has only $\mbZ^{S}_{\geq n}$-components (which is an analogue of \cite[Proposition 4.5]{HM03}) for every $n>0$.

\begin{example}[cf. {\cite[7.2]{HM03}}]\label{ex:MTMLie}
    Let $(\Gamma, R, S, \rho)$ be the data given in Example \ref{ex:MTM}. We can take representatives of finite-dimensional irreducible representations to be powers of the standard character, and they are Tate twists $\{ \mbQ_{p}(m) \}_{m \in \mbZ}$ when viewed as $\Gamma$-modules. The assumption of Theorem \ref{thm:wtdLie}(1) is that the cohomology group
    \[
    H^{1}_{\cts}(\Gamma, \mbQ_{p}(m)) 
    \cong 
    H^{1}_{\et}(\mbZ[1/p], \mbQ_{p}(m))
    \] is finite-dimensional for every $m>0$. Soul\'e \cite[page 376]{So81} proved that
    \[
    \dim_{\mbQ_{p}} H^{1}_{\et}(\mbZ[1/p], \mbQ_{p}(m))=
    \begin{cases}
        1 & (m=1,3,5, \dots) \\
        0 & (m=2,4,6, \dots)
    \end{cases}
    \] and
    $
    \dim_{\mbQ_{p}} H^{2}_{\et}(\mbZ[1/p], \mbQ_{p}(m))=0
    $ for every $m>1$. Hence we can apply Theorem \ref{thm:wtdLie} to conclude that the graded Lie algebra associated with the prounipotent radical of the weighted completion is freely generated by one element in each  degree $-2, -6,-10, \dots$. 
\end{example}

Let us consider the data associated with $K$ in Example-Definition \ref{ex:MCMM}. We can take representatives of finite-dimensional irreducible representations of $R=\mbG_{m}^{2}/\mu_{K}$ as
\[
\{ \mbQ_{p}(\boldsymbol{m}) \mid \boldsymbol{m}=(m_{1},m_{2}) \in \mbZ^{2}, m_{1} \equiv m_{2} \bmod w_{K} \}.
\] Here, by abuse of notation, $\mbQ_{p}(\boldsymbol{m})$ also denotes a one-dimensional representation of $\mbG_{m}^{2}/\mu_{K}$ on which the action of $\mbG_{m}$ through $w_{\p}$ (resp. $w_{\bar{\p}}$) is given by the $(-m_{1})$-th (resp. $(-m_{2})$-th) power of the standard character. Theorem \ref{thm:wtdLie} then gives us the following result:

\begin{theorem}\label{thm:MCCMLie}
   Let $(\Gamma, R, S, \rho)$ be the data associated with $K$, and write $\mathfrak{u}$ for the Lie algebra of the prounipotent radical of its weighted completion. Suppose that
   \[
   H^{2}_{\mathrm{cts}}(\Gamma, \mbQ_{p}(\boldsymbol{m}))
   \cong
   H^{2}_{\et}(\mcO_{K}[1/p], \mbQ_{p}(\boldsymbol{m}))
   \] vanishes for every $\boldsymbol{m} \in I_{K}$ ($I_{K}$ was defined at the beginning of \textsection \ref{sec:1.2}). Then the bigraded Lie algebra $\Gr^{W}_{\bullet} \mathfrak{u}$ is freely generated by one element in each bidegree in $-I_{K} \coloneqq \{ -\boldsymbol{m} \mid \boldsymbol{m} \in I_{K} \}$ and two elements in bidegree $(-1,-1)$.
\end{theorem}
\begin{proof}
    The assertion follows from Theorem \ref{thm:wtdLie} (3), together with Remark \ref{rmk:EPchar}.
\end{proof}

\section{Comparison of Galois image and weighted completion}\label{sec:3}

This section is devoted to establishing Theorem \ref{thm:main1}. Basically, we follow the strategy of Hain and Matsumoto. The proof proceeds in the following three steps, which are treated in \textsection \ref{sec:3.1}-\ref{sec:3.3}. The last subsection \ref{sec:3.4} discusses inert primes.

\begin{enumerate}
    \item We define a proalgebraic variant $\rho_{X,p}^{\un}$ of the pro-$p$ outer Galois representation by considering the unipotent completion of $\pi_{1}^{(p)}(X_{\bar{K}})$. By a result of Hain and Matsumoto, $\mfg_{X} \otimes \mbQ_{p}$ does not change even if we consider such a variant. 
    \item We consider a suitable quotient of the Zariski closure of the image of $\rho_{X,p}^{\un}$. We prove that the resulting representation is unramified outside $p$ and that the resulting quotient is a negatively weighted extension of $\mbG_{m}^{2}/\mu_{K}$. 
    \item Finally, applying Theorem \ref{thm:MCCMLie} yields the desired result. 
\end{enumerate}

\subsection{Galois action on prounipotent fundamental group}\label{sec:3.1}

Recall that $G_{K}$ comes equipped with the weight filtration induced by the pro-$p$ outer Galois representation $\rho_{X,p}$. In what follows, we consider another filtration on $G_{K}$ using its action on the continuous unipotent completion of the \'etale fundamental group of $X_{\bar{K}}$. We write $\pi_{1}^{\un}(X_{\bar{K}})$ for the continuous $\mbQ_{p}$-unipotent completion of $\pi_{1}^{\et}(X_{\bar{K}})$ \cite[A.2]{HM03}, and $\mathcal{P}(X_{\bar{K}})$ for its Lie algebra. We denote the lower central series of $\mathcal{P}(X_{\bar{K}})$ by
\[ 
\mcP(X_{\bar{K}}) \coloneqq L^{1}\mcP(X_{\bar{K}}) \supset \cdots \supset  L^{m+1}\mcP(X_{\bar{K}}) \coloneqq [L^{m}\mcP(X_{\bar{K}}), L^{1}\mcP(X_{\bar{K}})] \supset \cdots.
\] 
The abelianization of $\mcP(X_{\bar{K}})$ is naturally identified with $V_{p}(E)$, and the Lie algebra 
\[
\p(X_{\bar{K}})= \bigoplus_{m>0} \p(X_{\bar{K}})_{m} \coloneqq \bigoplus_{m>0} L^{m}\mcP(X_{\bar{K}})/L^{m+1}\mcP(X_{\bar{K}})
\] is a free graded Lie algebra of rank two over $\mbQ_{p}$. We can recover $\mcP(X_{\bar{K}})$ from $\p(X_{\bar{K}})$ by taking the completion with respect to the lower central series. We define proalgebraic groups $\Aut(\pi_{1}^{\un}(X_{\bar{K}}))$ and $\Out(\pi_{1}^{\un}(X_{\bar{K}}))$ by
\begin{align*}
&\Aut(\pi_{1}^{\un}(X_{\bar{K}})) \coloneqq \varprojlim_{m} \Aut(\mcP(X_{\bar{K}})/L^{m}\mcP(X_{\bar{K}})) \quad \text{and} \\
&\Out(\pi_{1}^{\un}(X_{\bar{K}})) \coloneqq \varprojlim_{m} \Out(\mcP(X_{\bar{K}})/L^{m}\mcP(X_{\bar{K}})),
\end{align*} respectively. Here $\Aut(\mcP(X_{\bar{K}})/L^{m}\mcP(X_{\bar{K}}))$ denotes the automorphism group scheme of the nilpotent Lie algebra $\mcP(X_{\bar{K}})/L^{m}\mcP(X_{\bar{K}})$, and $ \Out(\mcP(X_{\bar{K}})/L^{m}\mcP(X_{\bar{K}}))$ its quotient by the inner automorphism group \cite[Proposition A.8 and Corollary A.9]{HM03}. We have natural identifications
\begin{align*}
&\Lie \Aut(\pi_{1}^{\un}(X_{\bar{K}}))=\Der \mcP(X_{\bar{K}}) \quad \text{and} \quad \\
& \Lie \Out(\pi_{1}^{\un}(X_{\bar{K}}))=\ODer \mcP(X_{\bar{K}}) \coloneqq \Der \mcP(X_{\bar{K}}) / \Inn \mcP(X_{\bar{K}}),
\end{align*} where $\Der \mcP(X_{\bar{K}})$ denotes the Lie algebra of continuous derivations of $\mcP(X_{\bar{K}})$ and $\Inn \mcP(X_{\bar{K}})$ is that of inner derivations. We also have homomorphisms 
\begin{align}\label{eq:toabel}
\Aut(\pi_{1}^{\un}(X_{\bar{K}})) \to \GL(V_{p}(E)) \quad \text{and} \quad 
\Out(\pi_{1}^{\un}(X_{\bar{K}})) \to \GL(V_{p}(E))
\end{align} corresponding to the action on $V_{p}(E)$. The kernels of both homomorphisms in (\ref{eq:toabel}) are prounipotent \cite[Proposition A.8]{HM03}, and their Lie algebras are given by
\[
\Ker( \Der \mcP(X_{\bar{K}}) \to \End V_{p}(E))
\quad \text{and} \quad 
\Ker( \Der \mcP(X_{\bar{K}}) \to \End V_{p}(E))/\Inn \mcP(X_{\bar{K}}).
\] By the functoriality of unipotent completion, the pro-$p$ outer Galois representation $\rho_{X,p}$ induces a continuous homomorphism
\[
\rho_{X,p}^{\un} \colon G_{K} \to \Out(\pi_{1}^{\un}(X_{\bar{K}}))(\mbQ_{p}).
\]
Hence we obtain another filtration $\{ \mathcal{F}^{m}G_{K}  \}_{m>0}$ on $G_{K}$ by
\[
\mathcal{F}^{m}G_{K} \coloneqq \Ker(
    G_{K} \xrightarrow{\rho_{X,p}^{\un}}  \Out(\pi_{1}^{\un}(X_{\bar{K}}))(\mbQ_{p}) \to \Out(\mcP(X_{\bar{K}})/L^{m+1}\mcP(X_{\bar{K}}))
),
\] though the associated graded Lie algebra is known to be isomorphic to that of $\{ F^{m}G_{K}  \}_{m>0}$ after tensoring with $\mbQ_{p}$: 

\begin{proposition}[{\cite[Proposition 8.2]{HM03}}]\label{prp:profvsalg}
    We have $F^{m}G_{K} \subset \mathcal{F}^{m}G_{K}$ and the homomorphism
    \[
    \mfg_{X,m}=\Gr^{m}_{F} G_{K} \to \Gr^{m}_{\mcF} G_{K}
    \] has finite kernel and cokernel for every $m>0$. In particular, we have an isomorphism
    \[
    \mfg_{X} \otimes \mbQ_{p} \xrightarrow{\sim} \Gr^{\bullet}_{\mcF} G_{K} \otimes \mbQ_{p}
    \] between graded Lie algebras over $\mbQ_{p}$.
\end{proposition}

\subsection{Structure of certain outer automorphism group scheme}\label{sec:3.2}

First, we define a suitable subgroup of $\Out(\pi_{1}^{\un}(X_{\bar{K}}))$ containing the image of $\rho_{X,p}^{\un}$. Observe that the Galois image has the following three properties:

\begin{itemize}
    \item There is a natural homomorphism
    \begin{align}\label{eq:muK}
    \mu_{K}=\Aut(X) \to \Out(\pi_{1}^{\un}(X_{\bar{K}}))(\mbQ_{p})
    \end{align}
    by the functoriality of unipotent completion. Here, we regard $\mu_{K}$ as a constant closed subgroup of $\Out(\pi_{1}^{\un}(X_{\bar{K}}))$. The Galois image centralizes this subgroup.
    \item The decomposition $V_{p}(E)=V_{\p}(E) \oplus V_{\bar{\p}}(E)$ defines a closed immersion 
    \begin{align}\label{eq:decomp}
    \mbG_{m}^{2} \xrightarrow{\sim} \GL(V_{\p}(E)) \times  \GL(V_{\bar{\p}}(E)) \hookrightarrow  \GL(V_{p}(E)).
    \end{align} The Galois action  on $V_{p}(E)$ factors through this subgroup.
    \item Since we chose the Weierstrass tangential basepoint in \textsection \ref{sec:1.1}, we have the associated homomorphism
    \[
    \mbZ_{p}(1) \hookrightarrow \pi_{1}^{(p)}(X_{\bar{K}}),
    \quad \text{and hence} \quad 
    \mbG_{a} \hookrightarrow \pi_{1}^{\un}(X_{\bar{K}})
    \] by the functoriality of unipotent completion. The Galois action preserves the image of this map. Passing to Lie algebras, we obtain
    \begin{align*}
    \mbQ_{p} \cong \Lie(\mbG_{a}) \hookrightarrow \mcP(X_{\bar{K}}).
    \end{align*} Let $z$ be the image of $1 \in \mbQ_{p}$. Then the image of $z$ spans $L^{2}\mcP(X_{\bar{K}})/L^{3}\mcP(X_{\bar{K}})$.
\end{itemize}

\begin{definition}\label{dfn:asterisk}\,
    \begin{enumerate}
        \item For each $m>0$, let $\Aut^{\ast}(\mcP(X_{\bar{K}})/L^{m+1}\mcP(X_{\bar{K}}))$ be the closed subgroup of the automorphism group scheme of $\mcP(X_{\bar{K}})/L^{m+1}\mcP(X_{\bar{K}})$ that acts on $V_{p}(E)$ through the subgroup (\ref{eq:decomp}) and that preserves $\mbQ_{p} \cdot z \subset \mcP(X_{\bar{K}})/L^{m+1}\mcP(X_{\bar{K}})$. We define 
        \[
        \Aut^{\ast}(\pi_{1}^{\un}(X_{\bar{K}})) \coloneqq \varprojlim_{m} \Aut^{\ast}(\mcP(X_{\bar{K}})/L^{m}\mcP(X_{\bar{K}})),
        \] and write $\Der^{\ast}\mcP(X_{\bar{K}})$ for its Lie algebra. 
        \item We write $\Out^{\ast}(\pi_{1}^{\un}(X_{\bar{K}}))$ for the image of $\Aut^{\ast}(\pi_{1}^{\un}(X_{\bar{K}}))$ in  $\Out(\pi_{1}^{\un}(X_{\bar{K}}))$, and write $\ODer^{\ast}\mcP(X_{\bar{K}})$ for its Lie algebra. Moreover, we define $\Out^{\ast}_{\mu_{K}}(\pi_{1}^{\un}(X_{\bar{K}}))$ to be the centralizer of $\mu_{K}$ inside $\Out^{\ast}(\pi_{1}^{\un}(X_{\bar{K}}))$.
    \end{enumerate} 
\end{definition}

Note that $\Out^{\ast}_{\mu_{K}}(\pi_{1}^{\un}(X_{\bar{K}}))$ contains $\mu_{K}$. The quotient $\Out^{\ast}_{\mu_{K}}(\pi_{1}^{\un}(X_{\bar{K}}))/\mu_{K}$ surjects onto $\mbG_{m}^{2}/\mu_{K}$, and its kernel is prounipotent. The point is that it defines a negatively weighted extension, which allows us to apply results obtained in the previous section:

\begin{proposition}\label{prp:fundamental}
    The following assertions hold.
    \begin{enumerate}
        \item The image of $\rho_{X,p}^{\un} \colon G_{K} \to \Out(\pi_{1}^{\un}(X_{\bar{K}}))(\mbQ_{p})$ is contained in $\Out^{\ast}_{\mu_{K}}(\pi_{1}^{\un}(X_{\bar{K}}))(\mbQ_{p})$. 
        \item Both $\Aut^{\ast}(\pi_{1}^{\un}(X_{\bar{K}}))$ and $\Out^{\ast}(\pi_{1}^{\un}(X_{\bar{K}}))$ are negatively weighted extensions of $\mbG_{m}^{2}$ with respect to two central cocharacters $\{w_{\p}, w_{\bar{\p}} \}$  defined by $x \mapsto (x^{-1},1)$ and $x \mapsto (1,x^{-1})$, respectively (cf. Example-Definition \ref{ex:MCMM}).  
        \item $\Out^{\ast}_{\mu_{K}}(\pi_{1}^{\un}(X_{\bar{K}}))/\mu_{K}$ is a negatively weighted extension of $\mbG_{m}^{2}/\mu_{K}$ with respect to the two central cocharacters induced by $w_{\p}$ and $w_{\bar{\p}}$.
        \item The weight filtration on $\Der^{\ast}(\mcP(X_{\bar{K}}))$ is given by
        \begin{align*}
        &W_{-m}\Der^{\ast} \mcP(X_{\bar{K}})=\Ker \left( \Der^{\ast} \mcP(X_{\bar{K}}) \to  \Der^{\ast} \left( \mcP(X_{\bar{K}})/L^{m+1}\mcP(X_{\bar{K}}) \right) \right).
        \end{align*}
        Moreover, the map $\Der^{\ast} \mcP(X_{\bar{K}}) \to \ODer^{\ast} \mcP(X_{\bar{K}})$ induces an isomorphism
        \[
            \Der^{\ast} \mcP(X_{\bar{K}})/\mbQ_{p} \cdot z \xrightarrow{\sim} \ODer^{\ast}  \mcP(X_{\bar{K}})
        \] where $\mbQ_{p} \cdot z$ is regarded as a one-dimensional subspace of $\Inn  \mcP(X_{\bar{K}})$, from which one computes the weight filtration on $\ODer^{\ast} \mcP(X_{\bar{K}})$.
    \end{enumerate}
\end{proposition}

\begin{proof}
    The first assertion is immediate, and the third follows from the second. Fix a (topological) basis $\{ x,y \}$ of $\mcP(X_{\bar{K}})$ such that the images of $x$ and $y$ in $V_{p}(E)$ generate $V_{\p}(E)$ and $V_{\bar{\p}}(E)$, respectively, and $z$ is a nonzero scalar multiple of $[x,y]$. Such generators $x,y$ can be obtained by successively modifying arbitrary lifts of the chosen generators of $V_{\p}(E)$ and $V_{\bar{\p}}(E)$ by terms of higher degree in the lower central series. 
    
    Then $\{x, y \}$ induces a lift of the $\mbG_{m}^{2}$-action on $V_{p}(E)$ to $\mcP(X_{\bar{K}})$ preserving the line $\mbQ_{p} \cdot z$. This $\mbG_{m}^{2}$-action corresponds to the decomposition $\mcP(X_{\bar{K}})=\prod_{\boldsymbol{m}} \mcP(X_{\bar{K}})_{\boldsymbol{m}}$, where $\mcP(X_{\bar{K}})_{\boldsymbol{m}}$ is the subspace spanned by Lie monomials of bidegree $\boldsymbol{m}$ with respect to $x$ and $y$. Passing to Lie algebras, we obtain an exact sequence
    \[
    0 \to 
    \Der^{\ast}_{1}\mcP(X_{\bar{K}}) \to 
    \Der^{\ast}\mcP(X_{\bar{K}}) \to \Lie(\mbG_{m}^{2}) \to 0.
    \] Here, the Lie algebra $\Der^{\ast}_{1}\mcP(X_{\bar{K}})$ is given by
    \[
        \Der^{\ast}_{1}\mcP(X_{\bar{K}}) \coloneqq \{ d \in  \Der \mcP(X_{\bar{K}}) \mid d(\mcP(X_{\bar{K}})) \subset L^{2}\mcP(X_{\bar{K}}),\text{ and } d(z)=0 \},
    \] and it decomposes as the product of the following subspaces:
    \[
        \Der^{\ast}_{1}\mcP(X_{\bar{K}})_{-\boldsymbol{m}}=\{ d \in \Der^{\ast}_{1}\mcP(X_{\bar{K}}) \mid d(x) \in \mcP(X_{\bar{K}})_{\boldsymbol{m}+(1,0)},  d(y) \in \mcP(X_{\bar{K}})_{\boldsymbol{m}+(0,1)}   \} \quad (\boldsymbol{m} \in \mbZ^{2})
    \]  according to the action of $\mbG_{m}^{2}$. To prove that $\Aut^{\ast}(\pi_{1}^{\un}(X_{\bar{K}}))$ is a negatively weighted extension of $\mbG_{m}^{2}$ with respect to the two cocharacters, we must show $\Der^{\ast}_{1}\mcP(X_{\bar{K}})_{-\boldsymbol{m}}=\{ 0 \}$ for every $\boldsymbol{m}=(m_1, m_2) \not \in \mbZ^{2}_{>0}$. 

    First, suppose that $m_{1}$ is negative and let $d \in \Der^{\ast}_{1}\mcP(X_{\bar{K}})_{-\boldsymbol{m}}$. Then we have $d(y)=0$ since $\mcP(X_{\bar{K}})_{\boldsymbol{m}+(0,1)}$ is trivial, and $d(x)$ must be a multiple of $y$. However, since $d$ acts trivially on $V_{p}(E)$, we have $d(x)=0$. Hence $\Der^{\ast}_{1}\mcP(X_{\bar{K}})_{-\boldsymbol{m}}$ is trivial whenever $m_{1}$ is negative, and the case where $m_{2}$ is negative is similar. 

    Since every element of $\Der^{\ast}_{1}\mcP(X_{\bar{K}})$ acts trivially on $V_{p}(E)$ and we have 
    \[
    L^{2}\mcP(X_{\bar{K}})= \prod_{\substack{\boldsymbol{m} \in \mbZ^{2}_{>0}, \\ |\boldsymbol{m}| \geq 2}} 
    \mcP(X_{\bar{K}})_{\boldsymbol{m}},
    \] the $(0,0)$-component is also trivial. Now let $d$ be a (continuous) derivation contained in the $(-m,0)$-component for some $m>0$. We have $d(x)=0$ as $\mcP(X_{\bar{K}})_{(m+1,0)}$ is trivial. Moreover, since $d(y)$ has bidegree $(m,1)$, it is a multiple of $[x, [x, \dots, [x,y]]]$. Noting that $d(z)=0$, it follows that 
    \[
    d([x,y])=[d(x),y]+[x,d(y)]=[x, d(y)]=0.
    \] This implies $d(y)=0$ as desired since $\p(X_{\bar{K}})$ is a free graded Lie algebra of rank two generated by (the image of) $\{ x, y \}$. The vanishing of the $(0,-m)$-component for $m>0$ follows by interchanging $x$ and $y$.
    
    Since we have
    \[
    L^{m+1}\mcP(X_{\bar{K}})= \prod_{\substack{\boldsymbol{m} \in \mbZ^{2}_{>0}, \\ |\boldsymbol{m}| \geq m+1}}
    \mcP(X_{\bar{K}})_{\boldsymbol{m}}
    \] for every $m>0$, it follows that
    \[
    W_{-m} \Der^{\ast}\mcP(X_{\bar{K}})=\{ d \in  \Der^{\ast}\mcP(X_{\bar{K}}) \mid d(\mcP(X_{\bar{K}})) \subset L^{m+1}\mcP(X_{\bar{K}}) \}
    \] as desired. Since the centralizer of $\mbQ_{p} \cdot z$ in $\mcP(X_{\bar{K}})$ coincides with itself, the intersection of $\Inn \mcP(X_{\bar{K}})$ and $\Der^{\ast}_{1} \mcP(X_{\bar{K}})$ is given by $\mbQ_{p} \cdot z$. This proves the latter assertion of (4), and the latter assertion of (2)  follows from this.
\end{proof}

\begin{remark}
    The latter assertion of Proposition \ref{prp:fundamental} (4) may be regarded as a proalgebraic analogue of an argument used in \cite[(4.3)]{Na95} to lift the subgroup $\rho_{X,p}(F^{3}G_{K}) \subset \Out(\pi_{1}^{(p)}(X_{\bar{K}}))$ to $\Aut(\pi_{1}^{(p)}(X_{\bar{K}}))$.
\end{remark}

Write $\mathcal{G}_{K}$ for the Zariski closure of the image of the continuous homomorphism
\[
G_{K} \xrightarrow{\rho_{X,p}^{\un}}  \Out^{\ast}_{\mu_{K}}(\pi_{1}^{\un}(X_{\bar{K}})) (\mbQ_{p}) \xrightarrow{\bmod \mu_{K}} \left( \Out^{\ast}_{\mu_{K}}(\pi_{1}^{\un}(X_{\bar{K}}))/\mu_{K} \right) (\mbQ_{p}),
\] $\mathcal{U}_{K}$ for the prounipotent radical of $\mathcal{G}_{K}$ and $\mathfrak{u}_{K}$ for the Lie algebra of $\mathcal{U}_{K}$.

\begin{proposition}\label{prp:fundamental2}
    The following assertions hold.
    \begin{enumerate}
    \item The homomorphism
    $
        G_{K} \to \mathcal{G}_{K}(\mbQ_{p})
    $ is unramified outside $p$.
    \item The $G_{K}$-action on $V_{p}(E)$ induces a negatively weighted extension
    \[
    1 \to \mathcal{U}_{K} \to \mathcal{G}_{K} \to \mbG_{m}^{2}/\mu_{K} \to 1
    \] with respect to the two central cocharacters induced from $\{ w_{\p}, w_{\bar{\p}}\}$.
    \end{enumerate}
\end{proposition}

\begin{proof}
    The first assertion follows from \cite[Lemma 2.13]{Is23+}. To prove the second, let $\tilde{\mcG}_{K}$ be the Zariski closure of the image of $\rho_{X,p}^{\un}$ and $\tilde{\mcU}_{K}$ its prounipotent radical. Then (\ref{eq:toabel}) induces the exact sequence
    \[
    1 \to \tilde{\mcU}_{K} \to \tilde{\mathcal{G}}_{K} \to \mbG_{m}^{2} \to 1.
    \]
    Now we fix a splitting $s \colon \mbG_{m}^{2} \to \tilde{\mathcal{G}}_{K}$ (which exists by \cite[V\hspace{-.01em}I\hspace{-.15em}I\hspace{-.15em}I, Theorem 4.3]{Ho81}). We claim that $s(\mu_{K})$ coincides with $\mu_{K} \subset \Out^{\ast}_{\mu_{K}}(\pi_{1}^{\un}(X_{\bar{K}}))$ defined in (\ref{eq:muK}). 

    Consider $\zeta \in s(\mu_{K})(\mbQ_{p})$. Then, there exists a unique element $\zeta' \in \mu_{K}(\mbQ_{p})$ such that $\zeta' \zeta^{-1}$ lies in the kernel of 
    \[
    \Out(\pi_{1}^{\un}(X_{\bar{K}})) \to \GL(V_{p}(E))
    \] which is prounipotent (and hence the group of its $\mbQ_{p}$-points is torsion-free). On the other hand, since $\zeta'$ commutes with $\zeta$, the product $\zeta' \zeta^{-1}$ is still a torsion element. Therefore, it follows that $\zeta=\zeta'$ and hence $\mu_{K}=s(\mu_{K})$. As a result, $\mathcal{G}_{K}$ and $\mathcal{U}_{K}$ are isomorphic to $\tilde{\mathcal{G}}_{K}/\mu_{K}$ and $\tilde{\mcU}_{K}$, respectively. We obtain the desired exact sequence, which is a negatively weighted extension by Proposition \ref{prp:fundamental} (3).
\end{proof}

\subsection{Comparison of Lie algebras}\label{sec:3.3}

We complete the proof of Theorem \ref{thm:main1} as follows: Let $\mathcal{G}$ be the weighted completion of the data associated with $K$, $\mathcal{U}$ its prounipotent radical and $\mathfrak{u}$ the Lie algebra of $\mathcal{U}$. By Proposition \ref{prp:fundamental2} (1) and (2), together with the universality of weighted completion, we obtain a homomorphism
    $
    \mathcal{G} \to \mathcal{G}_{K}
    $. Noting that $G_{K} \to \mathcal{G}_{K}(\mbQ_{p})$ has Zariski-dense image, this homomorphism is surjective, and so is 
    $
    \mathcal{U} \to \mathcal{U}_{K}
    $. Since the induced surjection $\mfu \twoheadrightarrow \mfu_{K}$ is strict \cite[Theorem 3.11]{HM03}, passing to graded Lie algebras yields a surjection
    \[
    \Gr^{W}_{\bullet} \mathfrak{u} \twoheadrightarrow \Gr^{W}_{\bullet} \mathfrak{u}_{K}.
    \] Now Theorem \ref{thm:main1} follows from the following lemma, Theorem \ref{thm:MCCMLie} and $\mfg_{X, 2}=0$ \cite[(4.4)]{Na95}.

    \begin{lemma}
        We have a bidegree-reversing isomorphism $\Gr^{W}_{\bullet} \mathfrak{u}_{K} \cong \mfg_{X} \otimes \mbQ_{p}$.
    \end{lemma}
    \begin{proof}
        Since \cite[Corollary 4.15]{HM03} can be generalized to our setting without difficulty, the map $F^{1}G_{K} \to \mathcal{U}(\mbQ_{p})$ is found to have Zariski-dense image. Hence the map $F^{1}G_{K} \to \mathcal{U}_{K}(\mbQ_{p})$ also has Zariski-dense image. The assertion now follows from Proposition \ref{prp:profvsalg}, Proposition \ref{prp:fundamental} (4) and \cite[Lemma 7.5]{HM03}.
    \end{proof}

    \begin{remark}\label{rmk:constrained}
        Let $(\Gamma,R, S, \rho)$ be the data given in Example \ref{ex:MTM}. Hain and Matsumoto consider another weighted completion, which they call the \emph{constrained weighted completion} \cite[9.3]{HM03}. They have shown that the associated graded Lie algebra of the prounipotent radical of the constrained weighted completion is freely generated by one element in each degree $-6, -10, \dots$, cf. Example \ref{ex:MTMLie}. In particular, its degree $(-2)$-part is trivial. 

        Similarly, one can also construct the constrained weighted completion of the data associated with $K$, with the constraint that the representations be crystalline at both primes above $p$. Then an argument analogous to that in \cite[9.3]{HM03} shows that the associated (bi)graded Lie algebra $\Gr^{W}_{\bullet} \mfu_{\mathrm{c}}$\footnote{Here, the subscript c just stands for ``constrained''.} of the prounipotent radical $\mcU_{\mathrm{c}}$ is freely generated by one element in each bidegree in $-I_{K}$ under the same assumption of  Theorem \ref{thm:MCCMLie}. Since $\mathfrak{g}_{X, 2}=0$, it follows that the surjective map
        \[
        \Gr^{W}_{\bullet} \mfu \to \Gr^{W}_{\bullet} \mfu_{K}
        \] factors through $\Gr^{W}_{\bullet} \mfu_{\mathrm{c}}$. Therefore, under the assumption of Theorem \ref{thm:main1}, Conjecture \ref{cnj:ellDI} is equivalent to saying that this map is an isomorphism. Noting that $\mfu_{K}$ is embedded in the outer derivation algebra of $\mcP(X_{\bar{K}})$, this amounts to saying that $\mcU_{c} \to \Out(\pi_{1}^{\un}(X_{\bar{K}}))$ is a closed immersion. As was mentioned in Remark \ref{rmk:nonfree}, when $K$ is not a cyclotomic field,  Conjecture \ref{cnj:ellDI} does not hold. It therefore seems to be an interesting question to characterize the kernel of this homomorphism.
    \end{remark}

We close this subsection by discussing the \emph{local analogue} of Theorem \ref{thm:main1}. Namely, fix a prime of $\bar{\mbQ}$ lying above $\p$. Such a prime determines a homomorphism $
    G_{\mbQ_{p}} \to G_{K}
    $, and let $F^{m}G_{\mbQ_{p}}$ be the inverse image of $F^{m}G_{K}$. Similarly to the global case, we define the graded Lie subalgebra $\oplus_{m>0} \Gr^{m}_{F}G_{\mbQ_{p}}$ of $\mfg_{X}=\oplus_{m>0} \Gr^{m}_{F}G_{K}$. The point is that, in the local situation, we can prove the vanishing of the cohomology groups in question:

\begin{lemma}\label{lmm:localH2}
    For every $\boldsymbol{m} \in I_{K}$, we have
    \[
    \dim_{\mbQ_{p}} H^{1}(G_{\mbQ_{p}}, \mbQ_{p}(\boldsymbol{m}))=1 
    \quad \text{and} \quad
    H^{2}(G_{\mbQ_{p}}, \mbQ_{p}(\boldsymbol{m}))=\{0\}.
    \]
\end{lemma}
\begin{proof}
    The second equality follows from Tate's local duality \cite[Theorem 7.2.6]{NSW08}. Since the Euler--Poincar\'e characteristic is equal to $-1$, the first equality also follows.
\end{proof}

Therefore, arguing as in the proof of Theorem \ref{thm:main1}, we obtain the following result.

\begin{theorem}\label{thm:localanalogue}
    The bigraded Lie algebra $\oplus_{m>0} \Gr_{F}^{m}G_{\mbQ_{p}} \otimes \mbQ_{p}$ is generated by one element (which is possibly trivial) in each bidegree $\boldsymbol{m} \in I_{K}$. 
\end{theorem}

Here, we remark that, in the global case, we obtain an isomorphism
\[
(\mfg_{X} \otimes \mbQ_{p})^{\ab} \xrightarrow{\oplus \kappa_{\boldsymbol{m}}} \bigoplus_{\boldsymbol{m} \in I_{K}} \mbQ_{p}(\boldsymbol{m}), 
\] cf. the discussion before Remark \ref{rmk:upperbound}. Hence, Theorem \ref{thm:main1} gives a minimal generating set  of $\mfg_{X} \otimes \mbQ_{p}$. However, in the local case, we do not know whether the restriction of each character $\kappa_{\boldsymbol{m}}$ to the decomposition subgroup is nontrivial. The problem seems to be related to the nonvanishing of special values of Katz's two-variable $p$-adic $L$-functions, as they appear at least when considering the restriction of $\kappa_{(m,m)}$ to the inertia subgroup at $\p$ \cite[Corollary 1.10]{IG25}. For this reason, it seems difficult to prove that the inclusion
\[
\bigoplus_{m>0} \Gr_{F}^{m}G_{\mbQ_{p}} \otimes \mbQ_{p}
\hookrightarrow 
\bigoplus_{m>0} \Gr_{F}^{m}G_{K} \otimes \mbQ_{p}=\mfg_{X} \otimes \mbQ_{p}
\] is an equality. On the other hand, if the Galois group of the maximal pro-$p$ extension of $K(p)$ unramified outside $p$ is \emph{of local type} \cite[Definition 10.9.6]{NSW08}, then this inclusion is an equality \cite[Lemma 2.13]{Is23+}. The property of being of local type is related to the $p$-indivisibility of various special values of Hecke $L$-functions by results of Yager \cite[Theorem 3]{Ya82} and Wingberg \cite[Theorem]{Wi90}.

\subsection{Summary of main result for inert primes}\label{sec:3.4}

The aim of this last subsection is simply to remark that an analogue of Theorem \ref{thm:main1} for an \emph{inert} prime $p$ can be obtained by a similar argument. Throughout this subsection, assume $p$ is inert in $K$. The $p$-adic Tate module $T_{p}(E) \otimes K_{p}$ over $K_{p}$ decomposes into two one-dimensional Galois representations, which we denote by $V_{1}$ and $V_{2}$. We write
$
K_{p}(\boldsymbol{m}) \coloneqq V_{1}^{\otimes m_{1}} \otimes V_{2}^{\otimes m_{2}}
$ for $\boldsymbol{m}=(m_{1},m_{2}) \in \mbZ^{2}$. Each graded component $\mfg_{m} \otimes K_{p}$ is isomorphic to a finite sum of copies of $K_{p}(\boldsymbol{m})$, where $\boldsymbol{m}$ is an element of $I_{K}$ satisfying $|\boldsymbol{m}|=m$. 

\begin{theorem}\label{thm:main2}
    Assume that the second \'etale cohomology group
    \[
    H^{2}_{\et}(\mcO_{K}[1/p], K_{p}(\boldsymbol{m}))= \varprojlim_{n}  H^{2}_{\et}(\mcO_{K}[1/p], \mcO_{K}/p^{n}(\boldsymbol{m})) \otimes \mbQ_{p}
    \] is trivial for every $\boldsymbol{m} \in I_{K}$. Then the bigraded Lie algebra $\mfg_{X} \otimes K_{p}$ is generated by $\{ \sigma_{\boldsymbol{m}} \}_{\boldsymbol{m} \in I_{K}}$, where $\sigma_{\boldsymbol{m}}$ is an element of the $K_{p}(\boldsymbol{m})$-isotypic component of $\mfg_{X,|\boldsymbol{m}|} \otimes K_{p}$. 
\end{theorem}

\begin{remark}
    As in Remark \ref{rmk:Jannsen}, the cohomology groups in question are known to vanish in several cases. For example, consider the Galois group of the maximal pro-$p$ extension of $K(p)$ unramified outside $p$. Then, by \cite[Theorem 10.7.13]{NSW08}, it is a free pro-$p$ group if $p$ is prime to the class number of $K(p)$. For such an inert prime, the concerned cohomology groups vanish.
\end{remark}

The proof of Theorem \ref{thm:main2} is the same as that of Theorem \ref{thm:main1}, noting that we have to consider $\Res_{K_{p}/\mbQ_{p}}(\mbG_{m})/\mu_{K}$ instead of $\mbG_{m}^{2}/\mu_{K}$ in the data associated with $K$ (Example-Definition \ref{ex:MCMM}). Moreover, the two cocharacters are defined over $K_{p}$, though their product is still defined over $\mbQ_{p}$. Hence, we have to modify the content of \textsection \ref{sec:2} according to these differences, and we leave the details to the interested reader. 

On the other hand, at the time of writing, we do not have an analogue of Theorem \ref{thm:Ker} when $p$ is inert. We hope to revisit this problem in the near future.

\subsection*{Acknowledgements}\label{Acknowledgements} This paper is partly based on the author's talk at the Spring seminar on Arithmetic Galois theory in Toyonaka 2025, which was held soon after the workshop ``Low dimensional topology and number theory XVI'' in honor of Professor Hiroaki Nakamura's 60th birthday. The author would like to thank \emph{Benjamin Collas}, \emph{Hiroaki Nakamura} and \emph{Pierre D\`ebes} for giving him the opportunity to give a talk at the seminar, and \emph{Akio Tamagawa} for helpful discussions. The author also sincerely thanks the referee for carefully reading the manuscript and providing constructive comments. The author was supported by JSPS KAKENHI Grant Number JP23KJ1882.

\bibliographystyle{amsalpha}
\bibliography{references}

\end{document}